\documentclass[12pt,twoside]{amsart}
\usepackage{amsmath, amsthm, amscd, amsfonts, amssymb, graphicx}
\usepackage[bookmarksnumbered, plainpages]{hyperref}

\usepackage{mathrsfs}

\newtheorem{thm}{Theorem}[section]

\newtheorem{lem}[thm]{Lemma}
\newtheorem{prop}[thm]{Proposition}

\numberwithin{equation}{section}

\begin{document}

\title{\bf Semi-stable twisted holomorphic vector bundles over Gauduchon manifolds}
\author{Zhenghan Shen}

\address{Zhenghan Shen\\School of Mathematics and Statistics\\
Nanjing University of Science and Technology\\
Nanjing, 210094,P.R. China\\}\email{mathszh@njust.edu.cn}

\subjclass[2020]{53C07, 57N16}
\keywords{approximate Hermitian-Einstein metric, twisted vector bundle, semi-stability, Gauduchon manifold}

\maketitle

\begin{abstract}
In this paper, we study the semi-stable twisted holomorphic  vector bundles over compact Gauduchon manifolds. By using Uhlenbeck--Yau's continuity method, we show that  the existence of approximate Hermitian--Einstein structure and  the semi-stability of twisted holomorphic vector bundles are equivalent over compact Gauduchon manifolds. As its application, we show that the Bogomolov type inequality is also valid for a semi-stable twisted holomorphic vector  bundle.

\end{abstract}

\vskip 0.2 true cm


\pagestyle{myheadings}
\markboth{\rightline {\scriptsize Z. Shen }}
         {\leftline{\scriptsize Semi-stable twisted holomorphic vector bundles over Gauduchon manifolds}}

\bigskip
\bigskip


\section{ Introduction}

Let $M$ be a complex manifold of dimension $n$ and $g$ a Hermitian metric with the associated K\"ahler form $\omega$. The Hermitian metric $g$ is called Gauduchon if
\begin{equation*}
    \partial\bar{\partial}\omega^{n-1}=0.
\end{equation*}
It has been proved by Gauduchon (\cite{Gau84}) that every Hermitian metric is conformal to a Gauduchon metric (uniquely up to scaling when $n\geq 2$). The Hermitian metric $\omega$ is said to be Astheno-K\"ahler if $\partial\bar{\partial}\omega^{n-2}=0$. The concept of Astheno-K\"ahler was introduced by Jost and Yau in \cite{JY93}. In this case, the second Chern number of a twisted vector bundle can be well defined. Throughout the paper, we will assume $\omega$ is Gauduchon if there is no additional emphasis.

Twisted sheaves were introduced by Giraud in \cite{Gi71}, and can be defined in several equivalent ways. It can be regarded as family of sheaves on an open covering of $M$ together with a twisted gluing, as sheaves of modules over an Azumaya algebra on $M$ (\cite{C00}), as sheaves over a gerb on $M$ (\cite{Cha98,Gi71,Wang12}), etc. These objects have found many applications in physics for the description of the so-called $B$-fields and the $K$-theory. We should mention that the twisted $K$-theory has appeared earlier in the mathematical physics literature, in the study of the quantum Hall effect (\cite{CHMM98,CHM99}).

 In recent years, the existence of Hermitian--Einstein has been extensively studied. The classical Donaldson--Uhlenbeck--Yau theorem says that a holomorphic vector bundle over compact K\"ahler manifolds admit Hermitian--Einstein metrics if it is stable. This result was first proved by Narasimhan and Seshadri (\cite{NS65}) for compact Riemannian surface by using the methods of algebraic geometry. Subsequently, Donaldson (\cite{Don83,Don85,Don87}) gave a new proof of the Narasimhan-Seshadri theorem, and generalized to algebraic surfaces and algebraic manifolds by using Hermitian--Yang--Mills flow. Uhlenbeck and Yau (\cite{UY86, UY89}) proved it for general compact K\"ahler manifold by using  continuity method.

The inverse of the problem which states that a holomorphic bundle carrying a Hermitian--Einstein metric must be polystable is also valid. It was proved by Kobayashi (\cite{Ko87}) and L\"ubke (\cite{Lu83}) independently. So there is a one-to-one correspondence bewteen the algebraic notion of stability and the existence of Hermitian--Einstein on holomorphic vector bundle. This is usually referred as to Hitchin--Kobayashi correspondence (see \cite{AG,BJS14,iR,LT06,SZZ} and references therein). There are several interesting generalizations for this coorespondence. For example, Li--Yau (\cite{LY87}) studied the coorespondence for Hermitian manifolds with Gauduchon metric; Hitchin (\cite{Hit87}) and Simpson (\cite{Sim88}) studied the Higgs bundle case which plays an important role in non-abelian Hodge theory; Mochizuki (\cite{Mo20}) studied such correspondence on the non-compact K\"ahler manifolds; Biswas--Loftin--Stemmler (\cite{BLS13}) studied for affine manifolds; Biswas--Kasuya (\cite{BK21,BK21arx}) studied for Sasakian manifolds; Zhang--Zhang--Zhang (\cite{ZZZ21}) studied for some non-compact non-K\"ahler manifolds.

If one considers the weaker stability condition, there is an analog of such correspondence between the semi-stability and the approximate Hermitian-Einstein metric. Kobayashi (\cite{Ko87}) introduced the notion of approximate Hermitian--Einstein metric for holomorphic vector bundles. He proved a holomorphic vector bundle admitting an approximate Hermitian--Einstein metric must be semi-stable over a compact K\"ahler manifold. In
\cite{BG07}, Bruzzo and Gra\~na Otero generalized above result to the Higgs bundle case. When the base manifold is projective, Kobayashi proved the inverse part and conjectured that it should be true for general K\"ahler manifolds. It was solved by Li--Zhang (\cite{LZ15}) and Jacob (\cite{Jac14}) independently for the general K\"ahler manifold case. In the principal bundle case, it was studied by Biswas--Jacob--Stemmler (\cite{BBJS13,BJS12}). Recently, Nie--Zhang (\cite{NZ18}) proved the existence of approximate Hermitian--Einstein structures is equivalent to the semi-stability on Higgs bundles over compact Gauduchon manifolds. Subsequently, Zhang--Zhang--Zhang (\cite{ZZZ21}) also studied the semi-stable Higgs bundles on some non-compact Gauducon manifolds.

In the present paper, our considerations focus on the twisted vector bundles over compact Gauduchon manifolds. The twisted vector bundles are local vector bundles on a open covering with a twisted gluing, which is the point of view of C\u{a}ld\u{a}raru (\cite{C00, CKS03}). Let $\mathscr{U}:=\{U_i\}_{i\in I}$ be a  fixed open covering of $M$. A $B$-field on $M$ with respect to $\mathscr{U}$ is a family $B=\{B_{i}\}_{i\in I}$, where each $B_i$ is a $2$-form on $U_i$ such that
\begin{equation}
    B_i-B_j=d\beta_{ij},
\end{equation}
for some $1$-forms $\beta_{ij}$ on $U_{ij}:=U_i\cap U_j$. Notice that $d(\beta_{ij}+\beta_{jk}+\beta_{ki})=0$, we may find $U(1)$-valued functions $\alpha_{ijk}$ on $U_{ijk}:=U_i\cap U_j \cap U_k$ such that
\begin{equation}
    \beta_{ij}+\beta_{jk}+\beta_{ki}=-\alpha_{ijk}^{-1}d\alpha_{ijk},
\end{equation}
when the open covering $\mathscr{U}$ is sufficiently fine. Then $\alpha_B:=\{\alpha_{ijk}\}$ is a $2$-cocycle whose cohomology class lies in $H^2(M,\mathscr{O}_M^*)$, which is called the twist induced by $B$. We will omit the subscript $B$ when there is no confusion. An $\alpha$-twisted holomorphic vector bundle is a pair $E:=(\{E_i\}_{i\in I},\{\varphi_{ij}\}_{i,j\in I})$, where $E_i$ is a holomorphic vector bundle on $U_i$ and
\begin{equation*}
    \varphi_{ij}:E_j|_{U_{ij}}\to E_i|_{U_{ij}}
\end{equation*}
is an isomorphism of holomorphic vector bundles on $U_{ij}$, such that $\varphi_{ii}={\rm Id}_{E_i},\varphi_{ij}^{-1}=\varphi_{ji}$ and $\varphi_{ki}\circ \varphi_{jk}\circ \varphi_{ij}=\alpha_{ijk}\cdot {\rm Id}_{E|_{U_{ijk}}}$ for each $i,j,k\in I$. Here we usually omit the holomporphic structure $\bar{\partial}_{E_i}$ on $E_i$ when there is no confusion. It can be easily see that $1$-twisted holomorphic vector bundle is the usual holomorphic vector bundle.

In \cite{Wang12}, S. Wang proved the Hitchin--Kobayashi coorespondence for twisted holomorphic vector bundles over a gerbe on compact K\"ahler manifolds. All the definitions in \cite{Wang12} are in the category of twisted holomorphic vector bundles as holomorphic vector bundles over a gerb. Recently, A. Perego (\cite{Perego21}) generalized Wang's result to complex manifolds with a Gauduchon metric. And he also proved  the twisted version of approximate Hitchin--Kobayashi coorespondence on compact K\"ahler manifolds. Inspired by Perego's work, we consider the case that the approximate Hermitian--Einstein metric for twisted holomorphic vector bundle over compact Gauduchon manifolds. In fact, we prove the following theorem.

\begin{thm}\label{maintheorem}
Let $E=\{E_i,\varphi_{ij}\}_{i,j\in I}$ be an $\alpha$-twisted  holomorphic vector bundle of rank $r$ over a compact Gauduchon manifold $(M,\omega)$. Then $E$ is $\omega$-semi-stable if and only if it admits an approximate Hermitian--Einstein structure.
\end{thm}

We should remark that Perego (\cite{Perego21}) only proved the above theorem when the base manifold is a K\"ahler manifold. He mainly follow the methods of Jacob (\cite{Jac14}) and Kobayashi (\cite{Ko87}) by using Donaldson's heat flow. The proof relies on the properties of the Donaldson's Lagrangian. However, the Donaldson's Lagrangian is not well-defined when the base manifold is only Gauduchon. So his argument can not be generalized to Gauduchon case directly. In this paper,  we will adapt  Nie--Zhang's arguments (\cite{NZ18}) to our case, but there are some differences in the treatment of certain details.

As an application, we obtain the following Bogomolov type inequality.
\begin{thm}\label{thm2}
Let $(M,\tilde{\omega})$ be a compact Astheno-K\"ahler manifold of dimension $n$, and $\omega$ a Gauduchon metric which is conformal to $\tilde{\omega}$. Let $E=\{E_i,\varphi_{ij}\}_{i,j\in I}$ be an $\alpha$-twisted holomorohic vector bundle of rank $r$ over $M$. If $E$ is $\omega$-semi-stable, then we have the following Bogomolov type inequality
\begin{equation}
    \int_{M}\bigg(2c_2(E)-\frac{r-1}{r}c_1(E)\wedge c_1(E)\bigg)\wedge \frac{\tilde{\omega}^{n-2}}{(n-2)!}\geq 0.
\end{equation}
\end{thm}

The Bogomolov type inequality was first proved by Bogomolov (\cite{Bo79}) for semi-stable holomorphic vector bundles over complex algrebraic surfaces. Subsequently, it was generalized to Hermitian--Einstein vector bundles over compact K\"ahler manifolds by L\"ubke ( \cite{Lu82}). Recently, Li--Nie--Zhang (\cite{LNZ22}) proved the Bogomolov type inequality over Gauduchon Astheno-K\"ahler manifolds. In twisted case, Perego (\cite{Perego21}) obtained the Bogomolov type inequality for semi-stable $\alpha$-twisted holomorphic vector bundle over compact K\"ahler manifolds. In our paper, we will proved the Bogomolov type inequality by following Li--Nie--Zhang's arguments. 

This article is organized as follows. In Section \ref{sec:Pre}, we briefly present some notations for $\alpha$-twisted holomorphic vector bundle, such as connection, curvature and semi-stability. In Section \ref{sec:Pro}, we give the detailed proof of Theorem \ref{maintheorem}. In Section
\ref{sec:Bo}, we prove the Bogomolov type  inequality for semi-stable $\alpha$-twisted holomorphic vector bundles over Gauduchon Astheno-K\"ahler manifolds.

\section{Preliminary}\label{sec:Pre}
In this section, we will introduce the basic notations of $\alpha$-twisted vector bundles that will be used throughout the paper. We will follow the notations of \cite{Perego21}.
\subsection{Hermitian metrics, connections and curvatures}

Let $E=\{E_i,\varphi_{ij}\}_{i,j\in I}$ be an $\alpha$-twisted holomorphic vector bundle over compact Gauduchon manifold $(M,\omega)$. A Hermitian metric on $E$ is a connection $H=\{H_i\}_{i\in I}$, where $H_i$ is a Hermitian metric on $E_i$ and $H_i=\varphi_{ij}^TH_j\bar{\varphi}_{ij}$ for each $i,j\in I$, i.e., for any $\xi,\eta\in \Gamma(E_i)$,
\begin{equation*}
  H_i(\xi,\eta)=H_j(\varphi_{ij}(\xi),\varphi_{ij}(\eta)).
\end{equation*}
Since $\alpha_{ijk}$ is $U(1)$-valued, it can be easily seen that the definition is well-defined. And there always exists a Hermitian metric $H=\{H_i\}_{i\in I}$ on an $\alpha$-twisted vector bundle (see \cite{Perego21} for details).

We should remark that $H_i$ also denotes the matrix of smooth functions with respect to a given local frame and $\varphi_{ij}$ denotes the matrix of smooth functions with respect to the chosen local frames of $E_i$ and $E_j$ when there is no confusion.

A connection on $\alpha$-twisted vector bundle is a family $D=\{D_i\}_{i\in I}$, where $D_i$ is a connection on $E_i$ and locally the connection $1$-form $A_i$ satisfies
\begin{equation*}
    A_i=\varphi_{ij}^{-1}A_j\varphi_{ij}+\varphi_{ij}^{-1}d\varphi_{ij}+\beta_{ij}\cdot {\rm Id}_E.
\end{equation*}
The notion of connection on a twisted bundle can be found in \cite{Ka12}. We say that a connection $D=\{D_i\}_{i\in I}$ is compatible with the Hermitian metric $H=\{H_i\}_{i\in I}$, if each $D_i$ is compatible with $H_i$, i.e., for any $\xi,\eta\in \Gamma(E_i)$,
\begin{equation*}
    d(H_i(\xi,\eta))=H_i(D_i(\xi),\eta)+H_i(\xi,D_i(\eta)).
\end{equation*}
In a local given frame of $E_i$, it usually represents as
\begin{equation*}
    dH_i=A_i^TH_i+H_i\bar{A}_i.
\end{equation*}
Let $F_D\in \Omega^2(\mbox{End}(E))$ be  curvature form of the connection $D=\{D_i\}_{i\in I}$. According to \cite{Perego21}, we have
\begin{equation*}
    F_D|_{U_i}=F_{D_i}-B_i\cdot {\rm Id}_{E_i},
\end{equation*}
where $F_{D_i}$ is the curvature of $D_i$ for every $i\in I$. Thoughtout the paper, we assume that each $B_i$ of the $B$-filed $B$ is a $(1,1)$-form on $U_i$ and $\beta_{ij}$ is a $(1,0)$-form on $U_{ij}$. Then given a Hermitian metric $H$ on the $\alpha$-twisted holomorphic vector bundle $E$, there is a unique Chern connection $D_H$ which is compatible with the Hermitian metric $H$ and the holomorphic structure of $E$. In fact, we have the following lemma.
\begin{lem}[Lemma 2.25 in \cite{Perego21}]
Let $E$ be an $\alpha$-twisted holomorphic vector bundle and $H$ be a Hermitian metric on $E$.
\begin{enumerate}
    \item [(1)] There is a unique connection $D_H$ which is compatible both with the holomorphic structure of $E$ and the Hermitian metric $H$.
    \item [(2)] If each $B_i$ of the $B$-field  is a $(1,1)$-form for every $i\in I$, then $F_D\in \Omega^{1,1}({\mbox{End}(E)})$.
\end{enumerate}
\end{lem}

Let $F_H\in \Omega^{1,1}(\mbox{End}(E))$ be the curvature form of the Chern connection $D_H$, then we have
\begin{equation*}
    F_H|_{U_i}=F_{H_i}-B_i\cdot {\rm Id}_{E_i},
\end{equation*}
where $F_{H_i}$ is the curvature form of the Chern connection $D_{H_i}$ on $E_i$. A Hermitian metric $H=\{H_i\}_{i\in I}$ in $E=\{E_i,\varphi_{ij}\}_{i,j\in I}$ is called a  Hermitian--Einstein  metric if the curvature $F_H$ of the Chern connection $D_H$ satisfies
\begin{equation}
    \sqrt{-1}\Lambda_{\omega}F_H=\lambda\cdot {\rm Id}_E,
\end{equation}
where $\Lambda_{\omega}$ denotes the contraction of differential forms by $\omega$, and $\lambda$ is a real constant.

An $\alpha$-twisted holomorphic vector bundle $E=\{E_i,\varphi_{ij}\}_{i,j\in I}$ is said to be admitting an approximate Hermitian--Einstein structure if for every $\varepsilon>0$, there is a Hermitian metric $H_{\varepsilon}=\{H_{i,\varepsilon}\}_{i\in I}$ such that
\begin{equation}
    \max_{X}|\sqrt{-1}\Lambda_{\omega}F_{H_{\varepsilon}}-\lambda\cdot {\rm Id}_{E}|_{H_{\varepsilon}}<\varepsilon.
\end{equation}

\subsection{Stability}
Let $E=\{E_i,\varphi_{ij}\}_{i,j\in I}$ be an $\alpha$-twisted holomorphic vector bundle of rank $r$ and $H=\{H_i\}_{i\in I}$ be a Hermitian metric on $E$. Then the first Chern form of $(E,H)$ is
\begin{equation}
    c_1(E,H)=\frac{\sqrt{-1}}{2\pi}{\rm tr}(F_H)\in \Omega^{1,1}(M).
\end{equation}
The second Chern form of $(E,H)$ is
\begin{equation}
    c_2(E,H)=-\frac{1}{8\pi^2}\Big(({\rm tr}F_{H})^2-{\rm tr}(F_H\wedge F_H)\Big).
\end{equation}
The degree of $(E,H)$ with respect to the Gauduchon metric $\omega$ on $M$ is defined to be
\begin{equation}
    \deg_{\omega}(E):=\int_{M}c_1(E,H)\wedge \frac{\omega^{n-1}}{(n-1)!}.
\end{equation}
Since $\omega$ is a Gauduchon metric, the definition of the degree is well-defined and independent of the choice of the Hermitian metric $H$.

Let $\mathscr{F}=\{\mathscr{F}_i,\varphi_{ij}\}_{i,j\in I}$ be an $\alpha$-twisted coherent sheaf of rank $s$ on $M$. The definition of stability for twisted coherent sheaves was first introduced by Lieblich (\cite{Lieb07}) and  A. Perego (\cite{Perego19,Perego21}) studied the stability of twisted coherent sheaves in the language of twisted gluing of coherent sheaves. Given a rank $s$ $\alpha$-twisted coherent sheaves, the determinant of $\mathscr{F}=\{\det(\mathscr{F}_{i}),\det(\varphi_{ij})\}_{i,j\in I}$ is a locally free $\alpha^s$-twisted sheaf of rank $1$. The $\omega$-degree of $\mathscr{F}$ is defined by
\begin{equation}
    \deg_{\omega}(\mathscr{F}):=\deg_{\omega}(\det(\mathscr{F})).
\end{equation}
If $\mathscr{F}$ is non-trivial, the slope of $\mathscr{F}$ with respect to $\omega$ is defined by
\begin{equation}
    \mu_{\omega}(\mathscr{F})=\frac{\deg_{\omega}(\mathscr{F})}{{\rm rank}(\mathscr{F})}.
\end{equation}
Let $E=\{E_i,\varphi_{ij}\}_{i,j\in I}$ be an $\alpha$-twisted holomorphic vector bundle on $M$. We say an $\alpha$-twisted holomorphic vector bundle $E$ is $\omega$-stable (resp. $\omega$-semi-stable) if for every proper $\alpha$-twisted coherent subsheaf $\mathscr{F}$ of $E$, it holds that
\begin{equation}
    \mu_{\omega}(\mathscr{F})<\mu_{\omega}(E) \quad (\text{resp.}\quad \mu_{\omega}(\mathscr{F})\leq \mu_{\omega}(E)).
\end{equation}

\subsection{Endomorphisms of twisted vector bundles}
In the rest of this section, we will introduce the endomorphisms of  twisted vector bundles. Let $E=\{E_i,\varphi_{ij}\}_{i,j\in I}$ be an $\alpha$-twisted vector bundle on $M$. An endomorphism $f:E\to E$ is  a family $f=\{f_i\}_{i\in I}$, where $f_i:E_i\to E_i$ is an endomorphism of $E_i$ on $U_i$ and for every $i,j\in I$,
\begin{equation}
    \varphi_{ij}\circ f_i=f_j\circ \varphi_{ij}.
\end{equation}
It is easy to check that the eigenvalues of $f$ are smooth functions on the whole $M$. In fact, if $\lambda_i$ is a smooth function on $U_i$ and $\xi\in \Gamma(E_i)$ is  a nowhere vanishing smooth section such that $f_i(\xi)=\lambda_i\xi$, then
\begin{equation*}
    f_j(\varphi_{ij}(\xi))=\varphi_{ij}(f_j(\xi))=\varphi_{ij}(\lambda_i\xi)=\lambda_i\varphi_{ij}(\xi).
\end{equation*}
So $\lambda_i$ is an eigenvalue for $f_j$ with a nowhere vanishing smooth section $\varphi_{ij}(\xi)$. Hence the eigenvalues of $f_i$ glue together to give global functions on $M$.

Given a Hermitian metric $H=\{H_i\}_{i\in I}$ on $E=\{E_i,\varphi_{ij}\}_{i,j\in I}$, we say that an endomorphism $f$ is self-adjoint with respect to $H$, if for every $i\in I$ and $\xi,\eta\in \Gamma(E_i)$, it holds
\begin{equation}\label{selfadj}
    H_i(f_i(\xi),\eta)=H_i(\xi,f_i(\eta)).
\end{equation}
We denote
\begin{equation*}
    \text{Herm}(E,H)=\{f\in \mbox{End}(E)| f^{*H}=f\},
\end{equation*}
and
\begin{equation*}
    \text{Herm}^+(E,H)=\{f\in \text{Herm}(E,H)| f>0\},
\end{equation*}
where $f>0$ means that all the eigenvalues of $f$ are positive. If $H=\{H_i\}_{i\in I}$ and $K=\{K_i\}_{i\in I}$ are two Hermitian metrics on $\alpha$-twisted vector bundles, then for every $i\in I$ and $\xi,\eta\in \Gamma(E_i)$, the endomorphism $h=\{h_i\}_{i\in I}$  is defined by
\begin{equation*}
    H_i(\xi,\eta)=K_i(h_i\xi,\eta).
\end{equation*}
It is easy to see that $h=K^{-1}H$ is self-adjoint with respect to $H$ and $K$.

For any $h\in\text{Herm}^{+}(E,K)$, we can choose an open dense subset $W\subset M$  satisfying that, at each $x_0\in W$ there exists an open neighborhood $U_i$ of $x_0$, a local unitary basis $\{e_j\}_{j=1}^r$ with respect to $K_i$ and functions $\{\lambda_j\in C^{\infty}(U_i,\mathbb{R})\}_{i=1}^r$ such that
\begin{equation*}
    h(x)=\sum_{j=1}^re^{\lambda_j(x)}e_j(x)\otimes e^j(x)
\end{equation*}
for all $x\in U_i$, where $\{e^j\}_{j=1}^r$ is the dual basis corresponding to $\{e_j\}_{j=1}^r$. Then according to \cite[Lemma 2.56]{Perego21}, we have
\begin{equation*}
    \log h(x)=\sum_{j=1}^r\lambda_j(x)e_j(x)\otimes e^j(x).
\end{equation*}
Suppose $\Psi:\mathbb{R}\times \mathbb{R}\to \mathbb{R}$ is a smooth function, $A=\sum_{i,j=1}^rA_{ij}e^i\otimes e_j\in \mbox{End}(E)$, and $s\in \text{Herm}(E,K)$. Then we define $\Psi(s)(A)$ by
\begin{equation*}
    \Psi(s)(A)=\sum_{i,j=1}^n\Psi(\lambda_i,\lambda_j)A_{ij}e_i\otimes e^j.
\end{equation*}

\section{Proof of Theorem \ref{maintheorem}}\label{sec:Pro}

Let $E=\{E_i,\varphi_{ij}\}_{i,j\in I}$ be an $\alpha$-twisted holomorphic vector bundle of rank $r$ over a compact Gauduchon manifold $(M,\omega)$. Suppose $H=\{H_i\}_{i\in I}$ and $K=\{K_i\}_{i\in I}$ are two Hermitian metrics on $E$. Let $D_H=\{D_{H_i}\}_{i\in I}$ (resp. $D_K=\{D_{K_i}\}_{i\in I}$) be the Chern connection of $(E,H)$ (resp. $(E,K)$). Then for every $i\in I$, we have
\begin{equation}
    \partial_{H_i}=\partial_{K_i}+h_i^{-1}\partial_{K_i}h_i,
\end{equation}
where $\partial_{K_i}$ denotes the $(1,0)$-part of the connection induced by $D_{K_i}$ on $\mbox{End}(E_i)$. The relation between $F_{H_i}$ and $F_{K_i}$ is given by
\begin{equation}
    \begin{split}
        F_{H_i}&=F_{K_i}+\bar{\partial}_i(h_i^{-1}\partial_{K_i}h_i).
    \end{split}
\end{equation}
Then we have
\begin{equation}
    \begin{split}
        F_H|_{U_i}=&F_{H_i}-B_i\cdot {\rm Id}_{E_i}=F_{K_i}+\bar{\partial}_i(h_i^{-1}\partial_{K_i}h_i)-B_i\cdot {\rm Id}_{E_i}\\
        =&F_{K}|_{U_i}+\bar{\partial}_i(h_i^{-1}\partial_{K_i}h_i).
    \end{split}
\end{equation}
So if we glue together, it follows that
\begin{equation}
    F_H=F_K+\bar{\partial}(h^{-1}\partial_K h).
\end{equation}
Therefore $H$ solves the Hermitian--Einstein equation on $\alpha$-twisted holomorphic bundle $E$ if and only if there exist  Hermitian metric $K$ and $h\in \text{Herm}^+(E,K)$ such that
\begin{equation}
    \sqrt{-1}\Lambda_{\omega}F_K-\lambda \cdot {\rm Id}_E+\sqrt{-1}\Lambda_{\omega}\bar{\partial}(h^{-1}\partial_Kh)=0.
\end{equation}

Now, fixing a proper background Hermitian metric $K=\{K_i\}_{i\in I}$ on $\alpha$-twisted holomorphic vector bundle $E=\{E_i,\varphi_{ij}\}_{i,j\in I}$, we consider the following perturbed equation
\begin{equation}\label{pereq}
\begin{split}
    L_{\varepsilon}(h_{\varepsilon}):&=\sqrt{-1}\Lambda_{\omega}F_{H_{\varepsilon}}-\lambda \cdot {\rm Id}_{E}+\varepsilon \log(h_{\varepsilon})\\
    &=\sqrt{-1}\Lambda_{\omega}F_K-\lambda\cdot {\rm Id}_E+\sqrt{-1}\Lambda_{\omega}\bar{\partial}(h_{\varepsilon}^{-1}\partial_Kh_{\varepsilon})+\varepsilon\log(h_{\varepsilon}),
\end{split}
\end{equation}
where $h_{\varepsilon}=K^{-1}H_{\varepsilon}\in \text{Herm}^+(E,K)$ and $\varepsilon\in (0,1]$. It is obvious that $h_{\varepsilon}$ and $\log h_{\varepsilon}$ are self-adjoint with respect to $K$ and $H_{\varepsilon}$ in the sense of (\ref{selfadj}). In \cite{Perego21}, Perego proved that the equation (\ref{pereq}) is solvable for all $\varepsilon \in (0,1]$ by using the continuity method. If the $\alpha$-twisted holomorphic vector bundle $E$ is semi-stbale, we can show that
\begin{equation}\label{claim}
    \lim_{\varepsilon\to 0}\varepsilon \max_{M}|\log h_{\varepsilon}|_{K}=0.
\end{equation}
This implies that $\max_{M}|\sqrt{-1}\Lambda_{\omega}F_{H_{\varepsilon}}-\lambda \cdot {\rm Id}_{E}|\to 0$ as $\varepsilon\to 0$. So we obtain the existence of approximate Hermitian--Einstein structure on $E$.

Before giving the detailed proof of (\ref{claim}), we need the following two lemmas.

\begin{lem}[Lemma 5.7 in \cite{Perego21}]\label{lem3:1}
Let $E=\{E_i,\varphi_{ij}\}_{i,j\in I}$ be an $\alpha$-twisted holomorphic vector bundle on $M$, and $K$ be a Hermitian metric on $E$. If $h\in \text{Herm}^{+}(E,K)$ satisfies $L_{\varepsilon}(h)=0$ for some $\varepsilon>0$.
\begin{enumerate}
    \item [(1)]  We have
    \begin{equation}
        \frac{1}{2}\sqrt{-1}\Lambda_{\omega}\bar{\partial}\partial(|\log(h)|^2)+\varepsilon|\log(h)|_{K}^2\leq |\Phi(K)|_{K}\cdot |\log(h)|_{K},
    \end{equation}
    where $\Phi(K):=\sqrt{-1}\Lambda_{\omega}F_K-\lambda\cdot {\rm Id}_E$.
    \item[(2)] If $m:=\max_{M}|\log(h)|_{K}(x)$, then we have
    \begin{equation}
        m\leq \frac{1}{\varepsilon}\max_{x\in M}|\Phi(K)|_{K}(x).
    \end{equation}
    \item[(3)] There is a real number  $C$ (depending only on $\omega$ and $K$) such that
    \begin{equation}
        m\leq C(||\log h||_{L^2}+\max_{M}|\Phi(K)|_{K}).
    \end{equation}
\end{enumerate}
\end{lem}

\begin{lem}\label{keyid}
Let $E=\{E_i,\varphi_{ij}\}_{i,j\in I}$ be an $\alpha$-twisted holomorphic vector bundle on $M$, and $K=\{K_i\}_{i\in I}$ be a Hermitian metric on $E$. If $h_{\varepsilon}=\{h_{i,\varepsilon}\}_{i\in I}\in \text{Herm}^+(E,K)$ solves (\ref{pereq}) for some $\varepsilon$, then it holds that
\begin{equation}\label{keyid}
    \int_{M}{\rm tr}(\Phi(K)s_{\varepsilon})\frac{\omega^{n}}{n!}+\int_{M}\langle\Psi(s_{\varepsilon}(\bar{\partial}s_{\varepsilon}),\bar{\partial}s_{\varepsilon})\rangle_{K}\frac{\omega^n}{n!}=-\varepsilon ||s_{\varepsilon}||_{L^2}^2,
\end{equation}
where $s_{\varepsilon}=\{s_{i,\varepsilon}\}_{i\in I}=\log h_{\varepsilon}$, $\Phi(K)=\sqrt{-1}\Lambda_{\omega}F_K-\lambda\cdot {\rm Id}_E$ and
\begin{equation*}
    \Psi(x,y)=\begin{cases}
    \frac{e^{y-x}-1}{y-x},&x\neq y;\\
    1,&x=y.
\end{cases}
\end{equation*}
\end{lem}
\begin{proof}
Since $h_{\varepsilon}$ solves (\ref{pereq}), then we have
\begin{equation}
    \langle \Phi(K),s_{\varepsilon}\rangle_{K}+\langle \sqrt{-1}\Lambda_{\omega}\bar{\partial}(h_{\varepsilon}^{-1}\partial_Kh_{\varepsilon}),s_{\varepsilon}\rangle_K=-\langle s_{\varepsilon},s_\varepsilon\rangle.
\end{equation}
Integrating over $M$ gives
\begin{equation}\label{l3eq2}
    \int_{M}{\rm tr}(\Phi(K)s_{\varepsilon})\frac{\omega^n}{n!}+\int_{M}\langle \sqrt{-1}\Lambda_{\omega}\bar{\partial}(h_{\varepsilon}^{-1}\partial_Kh_{\varepsilon}),s_{\varepsilon}\rangle_K\frac{\omega^n}{n!}=-||s_{\varepsilon}||_{L^2}.
\end{equation}
Then comparing (\ref{l3eq2}) with (\ref{keyid}), it suffices to prove
\begin{equation}
    \int_{M}\langle \sqrt{-1}\Lambda_{\omega}\bar{\partial}(h_{\varepsilon}^{-1}\partial_Kh_{\varepsilon}),s_{\varepsilon}\rangle_K\frac{\omega^n}{n!}=\int_{M}\langle\Psi(s_{\varepsilon}(\bar{\partial}s_{\varepsilon}),\bar{\partial}s_{\varepsilon})\rangle_{K}\frac{\omega^n}{n!}.
\end{equation}
An easy calculation (see Proposition 3.1 in \cite{NZ18} for details) shows that
\begin{equation}\label{l3eq3}
   \int_{M}\langle \sqrt{-1}\Lambda_{\omega}\bar{\partial}(h_{\varepsilon}^{-1}\partial_Kh_{\varepsilon}),s_{\varepsilon}\rangle_K\frac{\omega^n}{n!}=\int_{M}\sqrt{-1}{\rm tr}(h_{\varepsilon}^{-1}\partial_Kh_{\varepsilon}\wedge \bar{\partial}s_{\varepsilon})\frac{\omega^{n-1}}{(n-1)!}.
\end{equation}

On the other hand, following the methods in \cite{LT06}, we can choose an open dense subset $W\subset M$ satisfying that at each $x\in W$ there exists an open neighborhood $U_i\in\mathscr{U}$ of $x$, a local unitary basis $\{e_j\}_{j=1}^r$ of $E_i$ with respect to $K_i$ and eigenvalues $\{\lambda_j\in C^{\infty}(U_i,\mathbb{R})\}_{j=1}^r$ such that
\begin{equation*}
    s_{i,\varepsilon}(y)=\sum_{j=1}^r\lambda_j(y)e_j(y)\otimes e^j(y),
\end{equation*}
for all $y\in U_i$, where $\{e^j\}_{j=1}^r$ is the dual basis of  $E_i^*$. Here $\lambda_1,\cdots, \lambda_r$ are eigenvalues of $s_{\varepsilon}$ (and hence of $s_{i,\varepsilon}$). Therefore, we have
\begin{equation*}
\partial_{K_i}h_{i,\varepsilon}(x)=e^{\lambda_j}\partial \lambda_je_j\otimes e^j+(e^{\lambda_k}-e^{\lambda_j})A^j_ke_j\otimes e^k,
\end{equation*}
and
\begin{equation*}
    \bar{\partial}s_{i,\varepsilon}(x)=\bar{\partial}\lambda_je_j\otimes e^j+(\lambda_k-\lambda_j)(-\overline{A_j^k})e_j\otimes e^k,
\end{equation*}
where $\{A_j^k\}_{j,k=1}^r$ are the $(1,0)$-forms defined defined by $\partial_{K_i}e_j=A_j^ke_k$. Then on each $U_i$, it follows that
\begin{equation*}
    \begin{split}
        &{\rm tr}\sqrt{-1}\Lambda_{\omega}(h_{i,\varepsilon}^{-1}\partial_{K_i}h_{i,\varepsilon}\wedge \bar{\partial}s_{i,\varepsilon})\\
        &=\sum_{j=1}^r|\bar{\partial}\lambda_j|^2+\sum_{j\neq k}(e^{\lambda_k}-e^{\lambda_j})(\lambda_j-\lambda_k)\sqrt{-1}\Lambda_{\omega}A^j_k \wedge (-\overline{A_j^k}) \\
        &=\sum_{j=1}^r|\bar{\partial}\lambda_j|^2+\sum_{j\neq k}\frac{e^{\lambda_k-\lambda_j}-1}{\lambda_k-\lambda_j}(\lambda_k-\lambda_j)^2|-\overline{A^j_k}|^2\\
        &=\sum_{j,k=1}^r\Psi(\lambda_j,\lambda_k)|(\bar{\partial}s_{i,\varepsilon})_j^k|^2.
    \end{split}
\end{equation*}
As both sides glue together, we have
\begin{equation*}
    {\rm tr}\sqrt{-1}\Lambda_{\omega}(h_{\varepsilon}^{-1}\partial_Kh_{\varepsilon}\wedge \bar{\partial}s_{\varepsilon})=\sum_{j,k=1}^r\Psi(\lambda_j,\lambda_k)|(\bar{\partial}s)_j^k|^2.
\end{equation*}
Using the similar argument, we can also prove
\begin{equation*}
    \langle\Psi(s)(\bar{\partial}s),\bar{\partial}s\rangle_K=\sum_{j,k=1}^r\Psi(\lambda_j,\lambda_k)|(\bar{\partial}s)_j^k|^2.
\end{equation*}
Therefore, we have
\begin{equation}\label{l3eq4}
     \langle\Psi(s)(\bar{\partial}s),\bar{\partial}s\rangle_K= {\rm tr}\sqrt{-1}\Lambda_{\omega}(h_{\varepsilon}^{-1}\partial_Kh_{\varepsilon}\wedge \bar{\partial}s_{\varepsilon}).
\end{equation}
Combining (\ref{l3eq3}) and (\ref{l3eq4}), we complete the proof.
\end{proof}

In the following, we are ready to prove Theorem \ref{maintheorem}. First, we show that the semi-stability implies the existence of approximate Hermitian--Einstein structure on $\alpha$-twisted holomorphic vector bundles. In fact, we have the following theorem.
\begin{thm}
Let $E=\{E_i,\varphi_{ij}\}_{i,j\in I}$ be an $\alpha$-twisted holomorphic vector bundle over $(M,\omega)$. If $E$ is $\omega$-semi-stable, then there is an approximate Hermitian--Einstein structure on $E$. i.e.,
\begin{equation}
    \max_{M}|\Phi(H_{\varepsilon})|_{H_{\varepsilon}}\to 0, \quad \text{as}  \quad \varepsilon \to 0.
\end{equation}
\end{thm}
\begin{proof}
Let $K$ be a background Hermitian metric on $E$. Without loss of generality, we can assume that the Hermitian $K$ satisfies
\begin{equation*}
    {\rm tr}(\sqrt{-1}\Lambda_{\omega}F_K-\lambda\cdot {\rm Id}_E)=0.
\end{equation*}
In fact, by an approximate conformal change $K=e^{\phi}\tilde{K}$, we have
\begin{equation}\label{t3eq1}
    \sqrt{-1}\Lambda_{\omega}\bar{\partial}\partial\phi=-\frac{1}{r}{\rm tr}(\sqrt{-1}\Lambda_{\omega}F_{\tilde{K}}-\lambda\cdot {\rm Id}_E),
\end{equation}
where $\tilde{K}$ is an arbitrary Hermitian metric on $E$. Since $\int_M{\rm tr}(\sqrt{-1}\Lambda_{\omega}F_{\tilde{K}}-\lambda\cdot {\rm Id}_E)\frac{\omega^n}{n!}=0$, then there is a smooth function $\phi$ satisfying (\ref{t3eq1}).

Let $\{h_{\varepsilon}\}_{0<\varepsilon\leq 1}$ be the solutions of the perturbed equation (\ref{pereq}) with the background Hermitian metric $K$. Then we have
\begin{equation*}
    ||\log h_{\varepsilon}||_{L^2}^2=-\frac{1}{\varepsilon}\int_{M}\langle \Phi(H_{\varepsilon}), \log h_{\varepsilon}\rangle_{H_{\varepsilon}}\frac{\omega^n}{n!},
\end{equation*}
where $\Phi(H_{\varepsilon})=\sqrt{-1}\Lambda_{\omega}F_{H_{\varepsilon}}-\lambda\cdot {\rm Id}_E$.
By \cite[Lemma 5.2]{Perego21}, we also have
\begin{equation}\label{t3eq2}
    \det(h_{\varepsilon})=1.
\end{equation}
We will consider the limit behavior of $||\log h_{\varepsilon}||_{L^2}$ in two cases.

{\bf Case 1}: $||\log h_{\varepsilon}||_{L^2}$ is bounded. i.e., there exists a constant $C_1$ such that
\begin{equation*}
    ||\log h_{\varepsilon}||_{L^2}<C_1<+\infty.
\end{equation*}
In this case, from Lemma \ref{lem3:1}, we have
\begin{equation*}
\begin{split}
    \max_{M}|\Phi(H_{\varepsilon})|_{H_{\varepsilon}}&=\varepsilon\cdot \max_{M}|\log h_{\varepsilon}|_{H_{\varepsilon}}<\varepsilon C\cdot(||\log h_{\varepsilon}||_{L^2}+\max_{M}|\Phi(K)|_{K})\\
    &\leq \varepsilon C(C_1+\max_M|\Phi(K)|_{K}) \to 0
    \end{split}
\end{equation*}
as $\varepsilon \to 0$.

{\bf Case 2}: $\limsup_{\varepsilon\to 0}||\log h_{\varepsilon}||_{L^2}=+\infty$. In this case, if $E=\{E_i,\varphi_{ij}\}_{i,j\in I}$ is  $\omega$-semi-stable, then we also have
\begin{equation}\label{t3eq3}
  \lim_{\varepsilon\to 0}\max_{M}|\Phi(H_{\varepsilon})|_{H_{\varepsilon}}=0.
\end{equation}
We will follow Simpson's argument (\cite{Sim88}) to show that if (\ref{t3eq2}) is not valid, there exists a coherent $\alpha$-twisted subsheaf contradicting the $\omega$-semi-stability of $E$.

We will prove (\ref{t3eq3}) by contradiction. If (\ref{t3eq3}) does not hold, then there exist $\delta>0$ and a subsequence $\varepsilon_i\to 0$, as $i\to +\infty$, such that
\begin{equation*}
    ||\log h_{\varepsilon_i}||_{L^2}\to +\infty
\end{equation*}
and
\begin{equation}\label{t3eq4}
    \max_{M}|\Phi(H_{\varepsilon_i})|_{H_{\varepsilon_i}}=\varepsilon_i\max_{M}|\log h_{\varepsilon_i}|_{H_{\varepsilon_i}}\geq \delta.
\end{equation}

Set
\begin{equation*}
    s_{\varepsilon_{i}}=\log h_{\varepsilon_{i}}, l_i=||s_{\varepsilon_{i}}||_{L^2}, u_{\varepsilon_{i}}=\frac{s_{\varepsilon_i}}{l_i}.
\end{equation*}
Then we have $||u_{\varepsilon_i}||_{L^2}=1$. From (\ref{t3eq2}), it follows that ${\rm tr}(u_{\varepsilon_i})=0$. Then combining (\ref{t3eq4}) with Lemma \ref{lem3:1}, we have
\begin{equation}\label{t3eq5}
    l_i\geq \frac{\delta}{C\varepsilon_i}-\max_{M}|\Phi(K)|_{K}
\end{equation}
and
\begin{equation}\label{t3eq6}
    \max_{M}|u_{\varepsilon_{i}}|<\frac{C}{l_i}\bigg(l_i+\max_{M}|\Phi(K)|_{K}\bigg)<C_2<+\infty.
\end{equation}

In the following, we divide our proof into two steps.

{\bf Step 1} We will show that $u_{\varepsilon_i}$ converge to $u_{\infty}$ weakly as $i\to +\infty$. We need to show that $||u_{\varepsilon_i}||_{L_1^2}$ are uniformly bounded. Since $||u_{\varepsilon_{i}}||_{L^2}=1$, it enough  to prove $||\bar{\partial} u_{\varepsilon_i}||_{L^2}$ are uniformly bounded.

By Lemma \ref{keyid}, for each $h_{\varepsilon_i}$, it holds
\begin{equation}\label{t3eq7}
    \int_{M}{\rm tr}(\Phi(K)u_{\varepsilon_i})\frac{\omega^n}{n!}+l_i\int_{M}\langle\Psi(l_iu_{\varepsilon_i})(\bar{\partial}u_{\varepsilon_i},\bar{\partial}u_{\varepsilon_i})\rangle_{K}\frac{\omega^n}{n!}=-\varepsilon_il_i.
\end{equation}
Combining (\ref{t3eq5}) and (\ref{t3eq7}), we obtain
\begin{equation}\label{t3eq8}
    \frac{\delta}{C}+\int_{M}{\rm tr}(\Phi(K)u_{\varepsilon_i})\frac{\omega^n}{n!}+l_i\int_{M}\langle\Psi(l_iu_{\varepsilon_i})(\bar{\partial}u_{\varepsilon_i},\bar{\partial}u_{\varepsilon_i})\rangle_{K}\frac{\omega^n}{n!}\leq \varepsilon_i\max_{M}|\Phi(K)|_{K}.
\end{equation}
Next, we consider the function
\begin{equation}\label{t3eq9}
    l\Psi(lx,ly)=\begin{cases}
    l,&x=y;\\
    \frac{e^{l(y-x)}-1}{y-x},&x\neq y.
    \end{cases}
\end{equation}
Because of (\ref{t3eq6}), we may assume that $(x,y)\in [-C_2,C_2]\times [-C_2,C_2]$.
It is easy to check that
\begin{equation}\label{t3eq10}
    l\Psi(lx,ly)\to \begin{cases}
    (x-y)^{-1},&x>y;\\
    +\infty,&x\leq y,
    \end{cases}
\end{equation}
increases monotonically as $l\to +\infty$. Let $\zeta\in C^{\infty}(\mathbb{R}\times \mathbb{R},\mathbb{R}^+)$ satisfying $\zeta(x,y)<(x-y)^{-1}$ whenever $x>y$. From (\ref{t3eq8}), (\ref{t3eq10}) and the same arguments in \cite[Lemma 5.4]{Sim88}, for $i$ large enough, we have
\begin{equation}\label{t3eq11}
    \frac{\delta}{C}+\int_{M}{\rm tr}(\Phi(K)u_{\varepsilon_i})\frac{\omega^n}{n!}+\int_{M}\langle\zeta(u_{\varepsilon_i})(\bar{\partial}u_{\varepsilon_i},\bar{\partial}u_{\varepsilon_i})\rangle_{K}\frac{\omega^n}{n!}\leq \varepsilon_i\max_{M}|\Phi(K)|_{K}.
\end{equation}
In particular, we take $\zeta=\frac{1}{3C_2}$. It can be easy check that $\frac{1}{3C_2}<\frac{1}{x-y}$ when $(x,y)\in [-C_2,C_2]\times [-C_2,C_2]$ and $x>y$. This implies that
\begin{equation*}
        \frac{\delta}{C}+\int_{M}{\rm tr}(\Phi(K)u_{\varepsilon_i})\frac{\omega^n}{n!}+\int_{M}\frac{1}{3C_2}|\bar{\partial}u_{\varepsilon_{i}}|_{K}^2\frac{\omega^n}{n!}\leq \varepsilon_i\max_{M}|\Phi(K)|_{K},
\end{equation*}
for $i>>0$. Then we obtain
\begin{equation*}
    \int_{M}|\bar{\partial}u_{\varepsilon_{i}}|_{K}^2\leq 3C_2^2\max_{M}|\Phi(K)|_{K}{\rm Vol}(M).
\end{equation*}
Therefore, $u_{\varepsilon_{i}}$ are bounded in $L_1^2$. Then we can choose a subsequence $\{u_{\varepsilon_{i_{j}}}\}$ (still denoted by $u_{\varepsilon_i}$ for simplicity) such that $u_{\varepsilon_i}\rightharpoonup u_{\infty}$ weakly in $L_1^2$. Since $L_1^2\hookrightarrow L^2$. it follows that
\begin{equation*}
    1=\int_{M}|u_{\varepsilon_i}|_{H_0}^2\to \int_{M}|u_{\infty}|_K^2.
\end{equation*}
This indicates that $||u_{\infty}||_{L^2}=1$ and $u_{\infty}$ is non-trivial. Then using (\ref{t3eq11}) and following a similar argument as in  \cite[Lemma 5.4]{Sim88}, we have
\begin{equation}\label{t3eq12}
    \frac{\delta}{C}+\int_{M}{\rm tr}(\Phi(K)u_{\infty})\frac{\omega^n}{n!}+\int_{M}\langle\zeta(u_{\infty})(\bar{\partial}u_{\infty},\bar{\partial}u_{\infty})\rangle_{K}\frac{\omega^n}{n!}\leq 0.
\end{equation}

{\bf Step 2} We will use Uhlenbeck and Yau's trick from \cite{UY86} to construct an $\alpha$-twisted coherent subsheaf which contradicts the $\omega$-semi-stability of $E$. Before going on, we need the following lemma.
\begin{lem}[Lemma 5.12 in \cite{Perego21}]\label{lem3:4}
Let $E=\{E_{i},\varphi_{ij}\}$ be an $\alpha$-twisted holomorphic vector bundle on $M$ whose associated locally free $\alpha$-twisted sheaf is $\mathscr{E}$. If $\pi\in L_{1}^2(\mbox{End}(E))$ is a weakly holomorphic $\alpha$-twisted subbundle of $E$, there is a coherent $\alpha$-twisted subsheaf $\mathscr{F}$ of $\mathscr{E}$ and an analytic subset $S$ of $M$ such that:
\begin{enumerate}
    \item [(1)] $S$ has codimension at least $2$ in $M$,
    \item [(2)] $\pi_{|M\setminus S}\in A^0(E_{|_{M\setminus S}})$ and have $\pi_{|{M\setminus S}}^*=\pi_{|M\setminus S}=\pi_{|M\setminus S}^2$ and $({\rm Id}_{E|M\setminus S}-\pi_{|M\setminus S})\circ \bar{\partial}(\pi_{|_{M\setminus S}})$,
    \item [(3)] $\mathscr{F}_{|_{M\setminus S}}$ is the image of $\pi_{|M\setminus S}$ and an $\alpha$-twisted holomorphic subbundle of $E$.
\end{enumerate}
\end{lem}

From (\ref{t3eq12}) and the technique in \cite[Lemma 5.5]{Sim88}, we conclude that the eigenvalues of $u_{\infty}$ are constant almost everywhere. Let $\lambda_1<\lambda_2<\cdots <\lambda_l$ be the distinct eigenvalues of $u_{\infty}$. Since ${\rm tr}(u_{\infty})={\rm tr}(u_{\varepsilon_i})=0$ and $||u_{\infty}||_{L^{2}}=1$, we conclude that $2\leq l\leq r$ be the distinct eigenvalues of $u_{\infty}$. Then for each eigenvalue $\lambda_j$ ($1\leq l-1$), which is a global function on $M$, we can construct a function
\begin{equation*}
    P_j:\mathbb{R}\to \mathbb{R}
\end{equation*}
such that
\begin{equation*}
    P_j=\begin{cases}
    1,&x\leq \lambda_j,\\
    0,&x\geq \lambda_j.
    \end{cases}
\end{equation*}
We denote $\pi_j=\{\pi_{i,j}\}_{i\in I}= P_j(u_{\infty})$. Following the same arguments as in \cite{Sim88}, we have
\begin{enumerate}
    \item [(i)] $\pi_j\in L_1^2(\mbox{End}(E))$;
    \item [(ii)] $\pi_j^2=\pi_j=\pi_j^{*K}$;
    \item [(iii)] $(Id_{E}-\pi_j)\circ \bar{\partial}\pi_j=0$.
\end{enumerate}
By Lemma \ref{lem3:4}, we know that the weakly holomorphic $\alpha$-twisted vector bundles $\{\pi_j\}_{j=1}^{l-1}$ determine $l-1$ coherent $\alpha$-twisted subsheaves of $E$. Denote $E_j=\{E_{i,j}\}_{i\in I}=\pi_j(E)$. Since ${\rm tr}(u_{\infty})=0$ and $u_{\infty}=\lambda_l\cdot {\rm Id}_E-\sum_{j=1}^{l-1}(\lambda_{j+1}-\lambda_j)\pi_j$, it holds that
\begin{equation}\label{t3eq13}
    \lambda_l{\rm rank}(E)=\sum_{j=1}^{l-1}(\lambda_{j+1}-\lambda_j){\rm rank}(E_j).
\end{equation}
Set
\begin{equation}\label{t3eq14}
    \nu=\lambda_l\deg(E)-\sum_{j=1}^{l-1}(\lambda_{j+1}-\lambda_j)\deg(E_j).
\end{equation}
On the one hand, substituting (\ref{t3eq13}) into (\ref{t3eq14}) directly, we have
\begin{equation}\label{t3eq15}
\nu=\sum_{j=1}^{l-1}(\lambda_{j+1}-\lambda_j){\rm rank}(E_j)\bigg(\frac{\deg(E)}{{\rm rank}(E)}-\frac{\deg(E_j)}{{\rm rank}(E_j)}\bigg).
\end{equation}
On the other hand, from the twisted Gauss--Codazzi equation (\cite{Perego21}), we have
\begin{equation}\label{t3eq16}
    \deg(E_j)=\int_{M}{\rm tr}(\pi_j\sqrt{-1}\Lambda_{\omega}F_K)-|\bar{\partial}\pi_j|^2_K\frac{\omega^n}{n!}
\end{equation}
Then substituting (\ref{t3eq16}) into (\ref{t3eq14}), we have
\begin{equation*}
    \begin{split}
        2\pi\nu&=\lambda_{l}\int_{M}{\rm tr}(\sqrt{-1}\Lambda_{\omega}F_K)\frac{\omega^n}{n!}\\
        & \quad -\sum_{j=1}^{l-1}(\lambda_{j+1}-\lambda_j)\int_{M}{\rm tr}(\pi_j\sqrt{-1}\Lambda_{\omega}F_K)-|\bar{\partial} \pi_j|_K^2\frac{\omega^n}{n!}\\
        &=\int_{M}{\rm tr}\bigg(\Big(\lambda_l\cdot{\rm Id}_E-\sum_{j=1}^{l-1}(\lambda_{j+1}-\lambda_j)\pi_{j}\Big)\sqrt{-1}\Lambda_{\omega}F_K\bigg)\frac{\omega^n}{n!}\\
        &\quad +\sum_{j=1}^{l-1}(\lambda_{j+1}-\lambda_j)\int_{M}|\bar{\partial}\pi_j|_K^2\frac{\omega^n}{n!}\\
        &=\int_{M}{\rm tr}(u_{\infty}\sqrt{-1}\Lambda_{\omega}F_K)\frac{\omega^n}{n!}+\int_{M}\bigg\langle \sum_{j=1}^{l-1}(\lambda_{j+1}-\lambda_j)(dP_{\alpha})^2(u_{\infty})(\bar{\partial}u_{\infty}),\bar{\partial}u_{\infty}\bigg\rangle_{K}\frac{\omega^n}{n!},
    \end{split}
\end{equation*}
where the functions $dP_{j}:\mathbb{R}\times \mathbb{R}\to \mathbb{R}$ are defined by
\begin{equation*}
    dP_{j}(x,y)=\begin{cases}
    \frac{P_j(x)-P_{j}(y)}{x-y}, &x\neq y;\\
    P'_j(x), &x=y.
    \end{cases}
\end{equation*}
By simple calculation, if $\lambda_a\neq \lambda_b$, we have
\begin{equation*}
    \sum_{j=1}^{l-1}(\lambda_{j+1}-\lambda_{j})(dP_j)^2(\lambda_{a},\lambda_{b})=|\lambda_a-\lambda_b|^{-1}.
\end{equation*}
Since ${\rm tr}(u_{\infty})=0$, according to (\ref{t3eq12}) and the same arguments in \cite{LZ15}, it follows that
\begin{equation}\label{t3eq17}
    \begin{split}
        2\pi v&=\int_{M}{\rm tr}(u_{\infty}\sqrt{-1}\Lambda_{\omega}F_K)\frac{\omega^n}{n!}+\int_{M}\bigg\langle \sum_{j=1}^{l-1}(\lambda_{j+1}-\lambda_j)(dP_{\alpha})^2(u_{\infty})(\bar{\partial}u_{\infty}),\bar{\partial}u_{\infty}\bigg\rangle_{K}\frac{\omega^n}{n!}\\
        &<-\frac{\delta}{C}.
    \end{split}
\end{equation}
Combining (\ref{t3eq15}) with (\ref{t3eq17}), we obtain
\begin{equation*}
    \sum_{j=1}^{l-1}(\lambda_{j+1}-\lambda_j){\rm rank}(E_j)\bigg(\frac{\deg(E)}{{\rm rank}(E)}-\frac{\deg(E_j)}{{\rm rank}(E_j)}\bigg)<0.
\end{equation*}
This indicates that there must exist a term $\mu(E)-\mu(E_{j_0})<0$, which contradicts the $\omega$-semi-stability of $E$.
\end{proof}

Finally, we prove the other direction of the theorem \ref{maintheorem}. In fact, we get the following theorem. The proof of the following theorem is standard and we present the proof for the readers' convenience.
\begin{thm}\label{thm3:5}
Let $(M,\omega)$ be  a compact Gauduchon manifold and $E=\{E_i,\varphi_{ij}\}_{i,j\in I}$ be an $\alpha$-twisted holomorphic vector bundle of rank $r$ over $M$. If $E$ admits an approximate Hermitian--Einstein structure, then $E$ is $\omega$-semi-stable.
\end{thm}

Before we giving the detailed proof, we need the following vanishing theorem.

\begin{prop}[Proposition 3.31 in \cite{Perego21}]\label{prop3:6}
Let $E_1$ and $E_2$ be  two $\alpha$-twisted holomorphic vector bundles  on $M$ of respective ranks $r_1$ and $r_2$, and suppose that
\begin{equation*}
    \frac{\deg_{\omega}(E_1)}{r_1}>\frac{\deg_{\omega}(E_2)}{r_2}.
\end{equation*}
If $E_1$ and $E_2$  admit approximate Hermitian--Einstein structure, then every morphism $f\in \text{Hom}(E_1,E_2)$ is zero.
\end{prop}

\begin{proof}[Proof of theorem \ref{thm3:5}]
Let $\mathscr{E}$  be the locally free $\alpha$-twisted coherent sheaf  associated to $E$ and $\mathscr{F}$ be any $\alpha$-twisted coherent subsheaf of $\mathscr{E}$ with rank $0<p<r$. Let $L$ be the $\alpha^p$-twisted holomorphic  line bundle corresponding to $\det(\mathscr{F})$. We consider the following untwisted holomorphic vector bundle
\begin{equation*}
   G= \wedge^pE\otimes L^{-1},
\end{equation*}
where $L^{-1}$ is the dual $\alpha^{-p}$-twisted line bundle of $L$. Following the same arguments as in \cite{Perego21}, we know that $G$ admits an approximate Hermitian--Einstein structure. Since there is a nontrivial holomorphic section of $G$, then by Proposition \ref{prop3:6}, we have
\begin{equation*}
    \deg_{\omega}(G)\geq 0.
\end{equation*}
Then
\begin{equation*}
\begin{split}
    0\leq \deg_{\omega}(G)&=\deg_{\omega}(\wedge^p E)-{\rm rank}(\wedge^pE)\deg_{\omega}(L)\\
    &=\deg_{\omega}(\wedge^p E)-{\rm rank}(\wedge^pE)\deg_{\omega}(\mathscr{F}).
    \end{split}
\end{equation*}
So we have
\begin{equation*}
    0\leq \mu_{\omega}(G)=\mu_{\omega}(\wedge^pE)-p\mu_{\omega}(\mathscr{F})=p(\mu_{\omega}(E)-\mu_{\omega}(\mathscr{F})).
\end{equation*}
This implies that $\mu_{\omega}(\mathscr{F})\leq \mu_{\omega}(E)$, i.e., $E$ is $\omega$-semi-stable.
\end{proof}

\section{Bogomolov type inequality for semi-stable $\alpha$-twisted holomorphic vector bundles}\label{sec:Bo}

In this section, we will prove the Bogomolov type inequality for semi-stable $\alpha$-twisted holomorphic vector bundles over compact Gauduchon Astheno-K\"ahler  manifolds. Before giving the detailed proof of Theorem \ref{thm2}, we need the following lemma.
\begin{lem}\label{lem4:1}
Let $(M,\tilde{\omega})$ be a compact Hermitian manifold of dimension $n$, and $\omega$ a Gauduchon metric which is conformal to $\tilde{\omega}$. Let $E=\{E_{i},\varphi_{ij}\}_{i,j\in I}$ be an $\alpha$-twisted holomorphic vector bundle of rank $r$ over $M$. If $E$ is $\omega$-semi-stable,  then for any $\varepsilon>0$, there exists a Hermitian metric $H_{\varepsilon}=\{H_{i,\varepsilon}\}_{i\in I}$ on $E$ such that 
\begin{equation}\label{sec:4eq1}
    \sup_{M}|\sqrt{-1}\Lambda_{\omega}F_{H_{\varepsilon}}^{\perp}|_{H_{\varepsilon}}<\varepsilon,
\end{equation}
where $F_{H_{\varepsilon}}^{\perp}=F_{H_{\varepsilon}}-\frac{1}{r}{\rm tr}F_{H_{\varepsilon}}\otimes {\rm Id}_{E}$ is the trace free part of $F_{H_{\varepsilon}}$.
\end{lem}
\begin{proof}
Since $\omega$ is conformal to $\tilde{\omega}$, then there exists a smooth function $\phi$ such that
\begin{equation*}
    \omega=e^{\phi}\tilde{\omega}.
\end{equation*}
Then by direct calculation, for any Hermitian metric $H$, we have
\begin{equation}\label{sec:4eq2}
\begin{split}
    \sup_{M}|\sqrt{-1}\Lambda_{\tilde{\omega}}F_H^{\perp}|_{H}&=\sup_{M}\bigg|\sqrt{-1}\Lambda_{\tilde{\omega}}F_H-\bigg(\frac{1}{r}\sqrt{-1}\Lambda_{\tilde{\omega}}{\rm tr}F_{H}\bigg)\otimes {\rm Id}_{E}\bigg|\\
    &=\sup_{M}e^{\phi}\bigg|\sqrt{-1}\Lambda_{\omega}F_H-\bigg(\frac{1}{r}\sqrt{-1}\Lambda_\omega{\rm tr}F_{H}\bigg)\otimes {\rm Id}_{E}\bigg|
    \end{split}
\end{equation}
Due to the fact that $\sqrt{-1}\Lambda_{\omega}F_H$ is self-adjoint with respect to $H$, we know that the eigenvalues of $\sqrt{-1}\Lambda_{\omega}F_H$ are real values. Then for any $\lambda_1,\cdots,\lambda_r\in \mathbb{R}, C\in \mathbb{R}$, we have the inequality
\begin{equation*}
    \sum_{j=1}^r(\lambda_j-\bar{\lambda})\leq \sum_{j=1}^r(\lambda_j-C)^2,
\end{equation*}
where $\bar{\lambda}=\frac{1}{r}\sum_{j=1}^r\lambda_j$. So we have
\begin{equation}\label{sec:4eq3}
    \bigg|\sqrt{-1}\Lambda_{\omega}F_H-\bigg(\frac{1}{r}\sqrt{-1}\Lambda_\omega{\rm tr}F_{H}\bigg)\otimes {\rm Id}_{E}\bigg|\leq |\sqrt{-1}\Lambda_{\omega}F_H-\lambda\cdot {\rm Id}_E|_H,
\end{equation}
where $\lambda=\frac{2\pi\deg_{\omega}(E)}{r{\rm Vol}(M,\omega)}$. Combining (\ref{sec:4eq2}) and (\ref{sec:4eq3}) yields that
\begin{equation}\label{sec:4eq4}
     \sup_{M}|\sqrt{-1}\Lambda_{\tilde{\omega}}F_H^{\perp}|_{H}\leq e^{\sup_{M}\phi}\sup_M|\sqrt{-1}\Lambda_{\omega}F_H-\lambda\cdot {\rm Id}_E|_H.
\end{equation}
Since $E$ is $\omega$-semi-stable, by Theorem \ref{maintheorem}, it admits an approximate Hermitian--Einstein structure on $E$. So for any $\varepsilon>0$, there exists a Hermitian metric $H_{\varepsilon}$ such that
\begin{equation}\label{sec4:eq5}
    \sup_{M}|\sqrt{-1}\Lambda_{\omega}F_{H_{\varepsilon}}-\lambda\cdot {\rm Id}_E|_{H_{\varepsilon}}<\varepsilon e^{-\sup_{M}\phi}.
\end{equation}
Combining (\ref{sec4:eq5}) with (\ref{sec:4eq4}) gives (\ref{sec:4eq1}).
\end{proof}

\begin{proof}[ Proof of theorem \ref{thm2}]Let $(M,\tilde{\omega})$ be a compact Astheno-K\"ahler manifold of dimension $n$, and $\omega$ a Gauduchon metric which is conformal to $\tilde{\omega}$. Let $E=\{E_i,\varphi_{ij}\}_{i,j\in I}$ be an $\alpha$-twisted holomorohic vector bundle of rank $r$ over $M$. Given any Hermitian metric $H=\{H_i\}_{i\in I}$, recall that
\begin{equation*}
    c_1(E,H)=\frac{\sqrt{-1}}{2\pi}{\rm tr}F_H, c_{2}(E,H)=-\frac{1}{8\pi^2}\Big(({\rm tr}F_H)^2-{\rm tr}(F_H^2)\big).
\end{equation*}
Since $F_H^{\perp}=F_H-\frac{1}{r}{\rm tr}F_H\otimes {\rm Id}_E$, one can easily get
\begin{equation*}
    {\rm tr}(F_H^2)=\frac{1}{r}({\rm tr}F_H)^2+{\rm tr}(F_{H}^{\perp}\wedge F_H^{\perp}).
\end{equation*}
Then we have
\begin{equation*}
    \begin{split}
        &\quad 4\pi^2\bigg(2c_{2}(E,H)-\frac{r-1}{r}c_1(E,H)\wedge c_1(E,H)\bigg)\\
        &=-\big(({\rm tr}F_H)^2-{\rm tr}(F_H^2)\big)+\frac{r-1}{r}({\rm tr}F_H)^2={\rm tr}(F_H^2)-\frac{1}{r}({\rm tr}F_H)^2\\
        &={\rm tr}(F_H^{\perp}\wedge F_H^{\perp}).
    \end{split}
\end{equation*}
By the Riemann bilinear relations, we have
\begin{equation}\label{sec4:eq6}
\begin{split}
     &\quad \int_{M}\bigg(2c_2(E)-\frac{r-1}{r}c_1(E)\wedge c_1(E)\bigg)\wedge \frac{\tilde{\omega}^{n-2}}{(n-2)!}\\
     &= \int_{M}\bigg(2c_2(E,H)-\frac{r-1}{r}c_1(E,H)\wedge c_1(E,H)\bigg)\wedge \frac{\tilde{\omega}^{n-2}}{(n-2)!}\\
     &=\frac{1}{4\pi^2}\int_{M}{\rm tr}(F_H^{\perp}\wedge F_H^{\perp})\wedge \frac{\tilde{\omega}^{n-2}}{(n-2)!}\\
     &=\frac{1}{4\pi^2}\int_{M}|F_H^{\perp}|_{H,\tilde{\omega}}-|\sqrt{-1}\Lambda_{\tilde{\omega}}F_{H}^{\perp}|_{H}^2\frac{\tilde{\omega^n}}{n!}.
     \end{split}
\end{equation}
Since $E$ is $\omega$-semi-stable, by Lemma \ref{lem4:1}, for every $\varepsilon>0$, there exists a Hermitian metric $H_{\varepsilon}$ such that
\begin{equation}\label{sec4:eq7}
    \sup_{M}|\sqrt{-1}\Lambda_{\tilde{\omega}}F_H^{\perp}|_{H_{\varepsilon}}\leq \varepsilon.
\end{equation}
Combining (\ref{sec4:eq6}) and (\ref{sec4:eq7}), and letting $\varepsilon \to 0$, we obtain
\begin{equation*}
\int_{M}\bigg(2c_2(E)-\frac{r-1}{r}c_1(E)\wedge c_1(E)\bigg)\wedge \frac{\tilde{\omega}^{n-2}}{(n-2)!}\geq 0.
\end{equation*}
\end{proof}

\medskip

\noindent{\bf Data Availability Statement:} Not applicable.

\medskip

\noindent {\bf  Acknowledgements:} The author would like to thank Prof. Xi Zhang and Dr. Pan Zhang for their useful discussions and helpful comments.
The research was supported by the National Key R and D Program of China 2020YFA0713100.
The  author is partially supported by NSF in China No.12141104 and 11721101.

\medskip

\noindent {\bf Conflicts of Interest:} The authors declare no conflicts of interest.



\medskip

\begin{thebibliography}{50}

\bibitem{AG} L. \'Alvarez-C\'onsul and O. Garc\'ia-Prada,
\newblock {\em Hitchin--Kobayashi correspondence, quivers, and vortices},
\newblock Comm. Math. Phys. {\bf238} (2003), 1-33.

\bibitem{BBJS13} I. Biswas, S. Bradlow, A. Jacob, M. Stemmler, \newblock{\em Approximate Hermitian-Einstein connections on principal bundles over a compact Riemann surface}, 
\newblock Ann. Global Anal. Geom. {\bf 44} (2013), no. 3, 257-268.

\bibitem{BJS12} I. Biswas, A. Jacob and M. Stemmler,
\newblock{\em Existence of approximate Hermitian-Einstein structures on semistable principal bundles}, 
\newblock Bull. Sci. Math. {\bf 136} (2012), no. 7, 745-751.

\bibitem{BJS14} I. Biswas, A. Jacob and M. Stemmler,
\newblock{\em The vortex equation on affine manifolds},
\newblock Trans. Amer. Math. Soc. {\bf 366} (2014), no. 7, 3925-3941.

\bibitem{BK21}
I. Biswas and H. Kasuya,
\newblock{\em Higgs bundles and flat connections over compact Sasakian manifolds},
\newblock Comm. Math. Phys. {\bf 385} (2021), no. 1, 267-290.

\bibitem{BK21arx}
I. Biswas and H. Kasuya,
\newblock{\em Higgs bundles and flat connections over compact Sasakian manifolds, II: quasi-regular bundles},
\newblock  arXiv:2110.10644v1, 2021.

\bibitem{BLS13} I. Biswas, J. Loftin and M. Stemmler,
\newblock{\em Flat bundles on affine manifolds},
\newblock Arab. J. Math. {\bf 2} (2013), no. 2, 159-175.

\bibitem{BG07}
U. Bruzzo and B. Gra\~na Otero,
\newblock{\em
Metrics on semistable and numerically effective Higgs bundles},
\newblock
J. Reine Angew. Math., {\bf 612} (2007), 59-79.

\bibitem{Bo79} F.A. Bogomolov,
\newblock{\em Holomorphic tensors and vector bundles on projective varieties},
\newblock Math. USSR Izvestija, {\bf13} (1979), no. 3 499-555.

\bibitem{C00} A. C\u{a}ld\u{a}raru,
\newblock{\em Derived categories of twisted sheaves on Calabi--Yau manifolds},
\newblock Ph.D. Thesis, Cornell University, 2000.

\bibitem{CHMM98} A.L. Carey, K.C. Hannabuss, V. Mathai, P. McCann,
\newblock{\em Quantum Hall effect on the hyperbolic plane},
\newblock Comm. Math. Phys. {\bf 190} (1998), no. 3, 629-673.

\bibitem{CHM99} A. Carey, K. Hannabuss, V. Mathai,
\newblock{\em Quantum Hall effect on the hyperbolic plane in the presence of disorder},
\newblock Lett. Math. Phys. {\bf 47} (1999), no. 3, 215-236.


\bibitem{CKS03}
A. C\u{a}ld\u{a}raru, S. Katz and E. Sharpe, \newblock{\em $D$-branes, $B$ fields, and Ext groups},
\newblock Adv. Theor. Math. Phys., {\bf 7} (2003), no. 3, 381-404.

\bibitem{Cha98}D. Chatterjee,
\newblock{\em On the construction of Abelian gerbs}, PhD Thesis, Cambridge (1998).

\bibitem{Don83} S.K. Donaldson,
\newblock{\em A new proof of a theorem of Narasimhan and Seshadri},
\newblock J. Differential Geom. {\bf 18} (1983), no. 2, 269-277.

\bibitem{Don85} S.K. Donaldson,
\newblock{\em Anti self-dual Yang--Mills connections over complex algebraic surfaces and stable vector bundles},
\newblock Proc. London Math. Soc., {\bf 50} (1985), no. 1, 1-26.

\bibitem{Don87} S.K. Donaldson,
\newblock{\em Infinite determinants, stable bundles and curvature},
\newblock Duke Math. J., {\bf 54} (1987), no. 1, 231-247.

\bibitem{Gau84}P. Gauduchon,
\newblock{\em La $1$-forme de torsion d'une vari\'et\'e hermitienne compacte},
\newblock Math. Ann. {\bf 267} (1984), no. 4, 495-518.

\bibitem{Gi71} J. Giraud, 
\newblock{\em Cohomologie non abélienne},
\newblock Die Grundlehren der mathematischen Wissenschaften, Band 179. Springer-Verlag, Berlin-New York, 1971.

\bibitem{Hit87} N.J. Hitchin,
\newblock{\em The self-duality equations on a Riemann surface},
\newblock Proc. London Math. Soc., {\bf 55} (1987), no. 1, 59-126.


\bibitem{iR} I.M. i Riera,
\newblock{\em A Hitchin--Kobayashi correspondence for K\"{a}hler fibrations},
\newblock J. Reine Angew. Math. {\bf 528} (2000), 41-80.

\bibitem{Jac14} A. Jacob,
\newblock{\em Existence of approximate Hermitian--Einstein structures on semi-stable bundles},
\newblock Asian J. Math., {\bf 18} (2014), no. 5, 859-883.

\bibitem{JY93} J. Jost and S.-T. Yau,
\newblock{\em A nonlinear elliptic system for maps from Hermitian to Riemannian manifolds and rigidity theorems in Hermitian geometry},
\newblock Acta Math. {\bf 170} (1993), no. 2, 221-254.

\bibitem{LNZ22} C. Li, Y.C. Nie and X. Zhang,
\newblock{Numerically flat holomorphic bundles over non-Kähler manifolds},
\newblock J. Reine Angew. Math. {\bf 790} (2022), 267-285.

\bibitem{LZ15}J.Y. Li and X. Zhang,
\newblock{\em Existence of approximate Hermitian--Einstein structures on semi-stable Higgs bundles},
\newblock Calc. Var. PDE, {\bf 52} (2015), no. 3-4, 783-795.

\bibitem{LY87} J. Li and S.-T. Yau,
\newblock{\em Hermitian--Yang--Mills connection on non-K\"{a}hler manifolds},
\newblock  Mathematical aspects of string theory, pp. 560-573,
Adv. Ser. Math. Phys., World Sci. Publishing, Singapore, 1987.

\bibitem{Lieb07} M. Lieblich,
\newblock{\em Moduli of twisted sheaves},
\newblock Duke Math. J. {\bf 138} (2007), no. 1, 23-118.

 \bibitem{Ka12} M. Karoubi,
\newblock{\em Twisted bundles and twisted $K$-theory},
\newblock Clay Math. Proc., {\bf 16} (2012), 223-257.

\bibitem{Ko87} S. Kobayashi,
\newblock{\em Differential geometry of complex vector bundles},
\newblock Publications of the Mathematical Society of Japan, 15. Kan\^o Memorial Lectures, 5. Princeton University Press, Princeton, NJ; Princeton University Press, Princeton, NJ, 1987.


\bibitem{Lu82} M. L\"ubke,
\newblock{\em Chernklassen von Hermite-Einstein-Vektorb\"undeln},
\newblock Math. Ann. {\bf 260} (1982), no. 1, 133-141.

\bibitem{Lu83} M. L\"ubke,
\newblock{\em Stability of Einstein--Hermitian vector bundles},
\newblock Manuscripta Math., {\bf 42} (1983), no. 2-3, 245-257.

\bibitem{LT06} M. L\"ubke and A. Teleman,
\newblock{\em The universal Kobayashi--Hitchin correspondence on Hermitian manifolds}, Mem. Amer. Math. Soc. {\bf 183} (2006), no. 863.

\bibitem{Mo20} T. Mochizuki,
\newblock{\em Kobayashi--Hitchin correspondence for analytically stable bundles},
\newblock Trans. Amer. Math. Soc., {\bf 373} (2020), no. 1, 551-596.

\bibitem{NS65}M.S. Narasimhan and C.S. Seshadri,
\newblock{\em Stable and unitary vector bundles on a compact Riemann surface},
\newblock Ann. of Math. {\bf 82} (1965), 540-567.

\bibitem{NZ18}Y.C. Nie and X. Zhang,
\newblock{\em
Semistable Higgs bundles over compact Gauduchon manifolds},
\newblock J. Geom. Anal., {\bf 28} (2018), no. 1, 627-642.

\bibitem{Perego19} A. Perego,
\newblock{\em K\"ahlerness of moduli spaces of stable sheaves over non-projective K3 surfaces},
\newblock Algebr. Geom. {\bf 6} (2019), no. 4, 427-453.


\bibitem{Perego21}A. Perego,
\newblock{\em Kobayashi--Hitchin correspondence for twisted vector bundles},
\newblock Complex manifolds, {\bf 8} (2021), no. 1, 1-95.


\bibitem{SZZ}Z.H. Shen, C.J. Zhang and X. Zhang,
\newblock{\em Flat Higgs bundles over non-compact affine Gauduchon manifolds},
\newblock J. Geom. Phys. {\bf175} (2022), 104475.

 \bibitem{Sim88} C.T. Simpson,
\newblock{\em Constructing variations of Hodge structure using Yang--Mills theory and applications to uniformization},
\newblock J. Amer. Math. Soc., {\bf 1} (1988), no. 4, 867-918.

\bibitem{UY86} K. Uhlenbeck and S.-T. Yau,
\newblock{\em On the existence of Hermitian-Yang-Mills connections in stable vector bundles},
\newblock Comm. Pure Appl. Math., {\bf 39} (1986), no. S, S257-S293.

\bibitem{UY89} K. Uhlenbeck and S.-T. Yau,
\newblock{\em  A note on our previous paper: On the existence of Hermitian--Yang--Mills connections in stable vector bundles},
Comm. Pure Appl. Math. {\bf 42} (1989), no. 5, 703-707.

\bibitem{Wang12} S. Wang,
\newblock{\em Objective $B$-fields and a Hitchin--Kobayashi correspondence},
\newblock Trans. Amer. Math. Soc., {\bf 364} (2012), no. 4, 2087-2107.

\bibitem{ZZZ21} C.J. Zhang, P. Zhang and X. Zhang,
\newblock{\em Higgs bundles over non-compact Gauduchon manifolds}, \newblock Trans. Amer. Math. Soc., {\bf 374} (2021), no. 5, 3735-3759.
\end{thebibliography}
\end{document}